\documentclass[10pt,a4paper]{amsart}

\usepackage{amssymb}
\usepackage[utf8]{inputenc}
\usepackage[english]{babel}
\usepackage{amsfonts,mathtools}
\usepackage{amsmath}
\usepackage{comment}
\usepackage{transparent}
\usepackage{graphicx}
\usepackage{xcolor}
\usepackage[normalem]{ulem}
\usepackage{color}

\DeclareFontFamily{U}{mathx}{\hyphenchar\font45}
\DeclareFontShape{U}{mathx}{m}{n}{
      <5> <6> <7> <8> <9> <10>
      <10.95> <12> <14.4> <17.28> <20.74> <24.88>
      mathx10
      }{}
\DeclareSymbolFont{mathx}{U}{mathx}{m}{n}
\DeclareFontSubstitution{U}{mathx}{m}{n}
\DeclareMathAccent{\widecheck}{0}{mathx}{"71}
\DeclareMathAccent{\wideparen}{0}{mathx}{"75}

\usepackage[bookmarksopen,bookmarksdepth=3,colorlinks,citecolor=red,pagebackref,hypertexnames=false]{hyperref}
\usepackage{amsthm}
\usepackage[nameinlink]{cleveref}

\usepackage{enumitem}
\usepackage{mathabx}

\newtheorem{main-theorem}{Theorem}
\newtheorem{proposition}{Proposition}[section]

\newtheorem{corollary}[proposition]{Corollary}
\newtheorem{lemma}[proposition]{Lemma}
\theoremstyle{remark}

\theoremstyle{definition}

\newtheorem*{acknowledgements}{Acknowledgements}

\DeclareMathOperator{\supp}{supp}

\newcommand{\R}{\mathbb{R}}
\newcommand{\C}{\mathbb{C}}

\newcommand{\N}{\mathbb{N}}
\newcommand{\Sph}{\mathbb{S}}
\newcommand{\dd}{\mathrm{d}}
\newcommand{\tm}{\mathrm{t}}
\newcommand{\x}{\mathrm{x}}
\newcommand{\Id}{\mathrm{Id}}
\newcommand{\Orth}{\mathrm{O}}
\newcommand{\T}{\mathsf{T}}

\def\ve{\varepsilon}

\def\ind{\mathbf 1}

\title[Time-independent potentials]{The initial-to-final-state inverse problem with 
time-independent potentials}

\author[M. Cañizares]{Manuel Cañizares}
\address[M.\ Cañizares]{Manuel Cañizares BCAM - Basque Center for Applied mathematics, 48009 Bilbao, Spain,}
\email{\href{mailto:mcanizares@bcamath.org}{\textrm{mcanizares@bcamath.org}}}

\author[P. Caro]{Pedro Caro}
\address[P.\ Caro]{Pedro Caro
Basque Center for Applied Mathematics and Ikerbasque (Basque Foundation for Science) Bilbao, Spain}
\email{\href{mailto:pcaro@bcamath.org}{\textrm{pcaro@bcamath.org}}}

\author[I. Parissis]{Ioannis Parissis}
\address[I.\ Parissis]{Departamento de Matem\'aticas, Universidad del Pa\'is Vasco, Aptdo. 644, 48080 Bilbao, Spain and Ikerbasque, Basque Foundation for Science, Bilbao, Spain}
\email{\href{mailto:ioannis.parissis@ehu.eus}{\textrm{ioannis.parissis@ehu.eus}}}

\author[T. Zacharopoulos]{Thanasis Zacharopoulos}
\address[T.\ Zacharopoulos]{Thanasis Zacharopoulos
BCAM - Basque Center for Applied Mathematics Bilbao, Spain}
\email{\href{mailto:azacharopoulos@bcamath.org}{\textrm{azacharopoulos@bcamath.org}}}

\begin{document}

\begin{abstract}The initial-to-final-state inverse problem consists in determining a 
quantum Hamiltonian assuming the knowledge of the state of the system at some fixed time, for every initial 
state. This problem was formulated by Caro and Ruiz and motivated by the data-driven 
prediction problem in quantum mechanics. {Caro and Ruiz} analysed the question of uniqueness for 
Hamiltonians of the form $-\Delta + V$ with an electric potential $V = V(\tm, \x)$ that 
depends on the time and space variables. In this context, {they} proved that 
uniqueness holds in dimension $n \geq 2$
whenever the potentials are bounded and have super-exponential decay at infinity. Although their result does not seem to be optimal, one 
would expect at least some degree of exponential decay to be necessary for the potentials. However, in this 
paper, we show that by restricting the analysis to Hamiltonians with time-independent 
electric potentials, namely $V = V(\x)$, uniqueness can be established for bounded integrable 
potentials  exhibiting only super-linear decay at infinity, in any dimension $n \geq 2$.
This surprising improvement is possible because, unlike 
Caro and Ruiz's approach, our argument avoids the use of complex geometrical optics (CGO). Instead, we 
rely on the construction of stationary states at different 
energies---this is possible because the 
potential does not depend on time. These states will have an explicit leading term,
given by a Herglotz wave, plus a correction term that will vanish as the energy grows. Besides the significant relaxation of decay assumptions on the potential, the avoidance of CGO solutions is important in its own right, since such solutions are not readily available in more complicated geometric settings.
\end{abstract}

\date{\today}


\maketitle

\section{Introduction}
In quantum mechanics, the family of wave functions 
$\{ u(t, \centerdot) : t \in [0, T] \} $ describes
the state of the system during an interval of time
$[0, T]$. If the motion takes place 
only under the influence of an electric potential $V$ and the initial 
state $u(0,  \centerdot)$ is prescribed by $f$, then states
\[u : (t, x) \in [0,T] \times \R^n \mapsto u(t, x) \in \C \]
are solutions of the initial-value problem for the Schr\"odinger equation
\begin{equation}
	\left\{
		\begin{aligned}
		& i\partial_\tm u = - \Delta u + V u & & \textnormal{in} \enspace (0,T) \times \R^n, \\
		& u(0, \centerdot) = f &  & \textnormal{in} \enspace \R^n.
		\end{aligned}
	\right.
	\label{pb:IVP}
\end{equation}
It is a well known fact
that, if $V \in L^1((0,T); L^\infty(\R^n))$, then for every $f \in L^2(\R^n)$
there exists a unique 
$u \in C([0,T]; L^2(\R^n))$ solving \eqref{pb:IVP}.
Moreover, the linear map 
$ f \in L^2(\R^n) \mapsto u \in C([0, T]; L^2(\R^n)) $ is bounded.
The solutions with {the} previous properties will be referred to as \emph{physical}.

Using the physical solutions associated to a potential $V$,
Caro and Ruiz formulated in \cite{zbMATH07801151} an
inverse problem consisting in determining the electric potential from data 
measured only at the initial and final times. These data were modelled by
the \textit{initial-to-final-state map}, which is defined by
\[\mathcal{U}_T : f \in L^2(\R^n) \mapsto u(T, \centerdot) \in L^2(\R^n),\]
where $u$ is the solution of the problem \eqref{pb:IVP}. This mapping is bounded
in $L^2 (\R^n)$. In {\cite{zbMATH07801151}},
Caro and Ruiz proved {that} $\mathcal{U}_T$ uniquely determines the potential $V = V(\tm, \x)$
in dimension $n \geq 2$ 
whenever $V \in L^1((0, T); L^\infty (\R^n))$ has super-exponential decay, that is,
$ e^{\rho |\x|} V \in L^\infty((0, T) \times \R^n) $ for all $\rho > 0$.
The reason that explains why the super-exponential decay was needed in \cite{zbMATH07801151}
is the use of a family of solutions, {usually called complex 
geometrical optics} (CGO) solutions, that grow exponentially {at infinity}. In light of \cite{zbMATH04103631,zbMATH00554321,zbMATH02208390}
one can expect that this result might be improved {by} only requiring
 {weaker} exponential decay of the following type: there exists an $\varepsilon > 0$ so that
$ e^{\varepsilon |\x|} V \in L^\infty((0, T) \times \R^n)$. However, it does not seem 
reasonable to hope to remove any kind of exponential decay \cite{zbMATH00850580}.

This paper is driven by two main objectives. First, we aim to analyze the initial-to-final-state inverse problem with the goal of eliminating any exponential decay assumptions on the potential. This objective is intrinsically motivated and forms part of a broader program to identify the widest class of potentials that allow for uniqueness. Second, we seek to establish a proof strategy that entirely avoids the use of CGO solutions. The motivation for this is that such solutions are not readily available in certain geometric settings where the uniqueness problem remains both significant and compelling, such as the $n$-dimensional torus, the unit sphere in  $\R^n$, or, more generally, an $n$-dimensional compact Riemannian manifold. 

Interestingly, the second objective supports the first: we successfully replace CGO solutions with the more well-behaved stationary state solutions, described below. This approach not only avoids the use of CGO solutions but also accommodates a significantly broader class of potentials, requiring only super-linear decay at infinity.

These results are the content of our main result, {Theorem~\ref{th:bounded}} below, which directly addresses a positive outcome to the first objective in the case of time-independent potentials.  The second objective—the avoidance of CGO solutions—is embedded within the proof of {Theorem~\ref{th:bounded}} and is explored throughout the paper. In order to present our result
it is convenient to introduce, for $V \in L^\infty (\R^n)$, the quantity
\[\vvvert V \vvvert = \sum_{j \in \N_0} 2^j \| V \|_{L^\infty (D_j)}, \]
which might be $\infty$, but otherwise it has the properties of a norm.
Here
\[ D_0 = \{ x \in \R^n : |x| \leq 1 \},\quad D_j = \{ x \in \R^n : 2^{j-1} < |x| \leq 2^j \}, \, \forall j \in \N. \]

\begin{main-theorem}\label{th:bounded}\sl  
Consider $V_1$ and $V_2$ in $L^1(\R^n) \cap L^\infty (\R^n)$ such that $  $
$\vvvert V_j \vvvert < \infty $ for $j \in \{1, 2 \}$.
Let $\mathcal{U}_T^1 $ and $ \mathcal{U}_T^2$ denote their
corresponding initial-to-final-state maps. Then,
\[ \mathcal{U}_T^1 = \mathcal{U}_T^2 \Rightarrow V_1 = V_2.\]
\end{main-theorem}

As anticipated, this result demonstrates that restricting the analysis to Hamiltonians with time-independent electric potentials allows for a significant relaxation of the required decay conditions on the potential.  This remarkable  improvement is possible by the freedom in the time variable, which allows us to adopt an approach based on constructing stationary states at high energies—--where energy can be interpreted as the Fourier-dual variable of time.

The scheme of the proof of \Cref{th:bounded} 
is an adaptation of the one followed by Caro and 
Ruiz in \cite{zbMATH07801151}. From the equality $\mathcal{U}_T^1 = \mathcal{U}_T^2$,
we derive an orthogonality relation
\begin{equation}\label{id:orthogonality_relation}
\int_{\R^n} (V_1 - V_2) w_1 w_2 \, = 0
\end{equation}
for $w_j$, ${j\in\{1,2\}}$, solution of $(\Delta + \lambda^2 - V_j) w_j = 0$ in $\R^n$.
For $\lambda$ sufficiently large we will construct stationary states in the form 
$w_j = u_j + v_j$, where $u_j$ is a Herglotz wave and the correction term $v_j$ is negligible
with respect to $u_j$ as $\lambda$ grows. Then, the idea is to choose the Herglotz waves
in such a way that $\int_{\R^n} (V_1 - V_2) u_1 u_2$
allows {us} to compute the Fourier transform of
$V_1 - V_2$ as the energy $\lambda^2$ tends to infinity, and derive from the orthogonality 
relation that {it} has to vanish. This would imply that $V_1 = V_2$.

Other inverse problems for the dynamical Schr\"odinger equation have been 
previously formulated and analysed in 
\cite{zbMATH01886353,zbMATH05549395,zbMATH05655673,zbMATH05839237,zbMATH06733553,
zbMATH06864429}. Additionally, time-dependent Hamiltonians have been also considered
in \cite{zbMATH05379127,zbMATH06516179,zbMATH06769718,zbMATH07033617,zbMATH07242805}.
In most of these articles, the authors have studied 
inverse problems with data given by a dynamical 
Dirichlet-to-Neumann map defined on the boundary of a domain where the 
non-constant parts of the Hamiltonians are confined. Note that 
any type of dynamical Dirichlet-to-Neumann map in our context would correspond to invasive 
data.

The contents of the paper are as follows. In \Cref{sec:stationary_states} we construct
the stationary states that will be plugged into
\eqref{id:orthogonality_relation}.  \Cref{sec:orthogonality} is devoted to prove the 
orthogonality relation. In \Cref{sec:uniqueness} we prove \Cref{th:bounded}.

\section{Stationary states}\label{sec:stationary_states}

In this section we construct the stationary states that will be plugged into the orthogonality
relation \eqref{id:orthogonality_relation}. We start by introducing some useful Banach spaces, 
then we provide some bounds for Herglotz waves, and we finish by constructing the 
stationary states as perturbations of Herglotz waves.

\subsection{Some useful spaces}
The Banach spaces $B(\R^n)$ and $B^\ast (\R^n)$ are defined through the norms
\[
\|f\|_{B(\R^n)}\coloneqq \sum_{j\in\mathbb N_0} 2^{j/2}\|f\|_{L^2(D_j)},\qquad \|f\|_{B^*(\R^n)}\coloneqq \sup_{j\in\mathbb N_0}\left(2^{-j/2}\|f\|_{L^2(D_j)}\right).
\]
It is not hard to see that $B^*$ is the dual of $B$, \cite[\S 14.1]{HormII}, and 
\[
B(\R^n)\subset L^2(\R^n)\subset B^*(\R^n),\qquad \|f\|_{B^* (\R^n)}\leq \|f\|_{L^2(\R^n)}\leq \|f\|_{B(\R^n)}.
\]

\subsection{Herglotz waves}
For $\lambda > 0$ and $f \in C^\infty (\Sph^{n - 1})$, we define the Herglotz wave
\[ E_\lambda f (x) = \int_{\Sph^{n - 1}} e^{-i \lambda x \cdot \theta} f (\theta) \, \dd \sigma (\theta) \quad \forall x \in \R^n. \]
Here $\sigma$ denotes the surface measure in $\Sph^{n - 1}$. 

\begin{lemma}\label[lemma]{lem:B*extension}\sl There exists a constant $C > 0$, that only might depend on $n$, 
such that
\[\| E_\lambda f \|_{B^\ast (\R^n)} \leq \frac{C}{\lambda^{(n - 1)/2}} \| f \|_{L^2(\Sph^{n - 1})}\]
for all $f \in C^\infty (\Sph^{n - 1})$ and $\lambda > 0$. Additionally, for all $\lambda > 0$ 
and $f \in C^\infty (\Sph^{n - 1})$ we have that $(\Delta + \lambda^2) E_\lambda f = 0$
in $\R^n$.
\end{lemma}

\begin{proof}
Start by noting that $E_\lambda f (x) = (2\pi)^{n/2} \widehat{f\, \dd \sigma} (\lambda x)$, where
$f\, \dd \sigma$ is the {compactly supported} distribution in $\R^n$ defined by
\[ \langle f\, \dd \sigma, \phi \rangle =  \int_{\Sph^{n - 1}} \phi f \, \dd \sigma, \quad \phi \in C^\infty (\R^n). \]
It is then clear that
\[ \| E_\lambda f \|_{L^2(D_j)} \leq \frac{(2\pi)^{n/2}}{\lambda^{n/2}} \| \widehat{f\, \dd \sigma} \|_{L^2(B(\lambda 2^j))}, \]
where $B(\lambda 2^j) = \{ y \in \R^n : |y| \leq \lambda 2^j \}$. Additionally, 
by the extension version of the 
trace theorem for the Fourier transform \cite[Theorem 7.1.26]{zbMATH01950198}, we have\[ \| \widehat{f\, \dd \sigma} \|_{L^2(B(\lambda 2^j))} \lesssim (2^j \lambda)^{1/2} \| f \|_{L^2(\Sph^{n - 1})}. \]
The implicit constant here only depends on $n$.
Furthermore, the fact that the identity
$(\Delta + \lambda^2) E_\lambda f = 0$ holds in $\R^n$ follows from a direct computation {using}
the definition of $E_\lambda$.
\end{proof}

\subsection{Construction of the correction term}
For $\lambda > 0$ and $f \in \mathcal{S} (\R^n)$, we define
\[ P_\lambda f (x) = \frac{1}{(2\pi)^{n/2}} \, \mathrm{p.v.} 
\int_{\R^n} \frac{e^{i x \cdot \xi}}{\lambda^2 - |\xi|^2} \widehat{f}(\xi) \, \dd \xi 
\quad \forall x \in \R^n. \]
Here $\mathrm{p.v.} \int_{\R^n}$ stands for principal value,
and {is defined as} the limit of the 
integrals over the sets $\{ \xi \in \R^n : |\lambda^2 - |\xi|^2| > \ve \}$ as
$\ve$ tends to $0$.

\begin{lemma}\label[lemma]{lem:p.v.}\sl
There exists a constant $C > 0$, that {depends only} on $n$, such that
\[\| P_\lambda f \|_{B^\ast (\R^n)} \leq \frac{C}{\lambda} \| f \|_{B(\R^n)}\]
for all $f \in \mathcal{S} (\R^n)$ and $\lambda > 0$.
Additionally, for every $\lambda > 0$ and every $f \in B(\R^n)$ we have that 
$(\Delta + \lambda^2) P_\lambda f = f$ in $\R^n$.
\end{lemma}

\begin{proof}
The inequality is a particular case of \cite[Theorem 6.3]{zbMATH03524039}, while the identity
$(\Delta + \lambda^2) P_\lambda f = f$ in $\R^n$ holds by the definition of $P_\lambda$ for
$f \in \mathcal{S} (\R^n)$. Then, by the density of $\mathcal{S} (\R^n)$ in $B(\R^n)$ and the boundedness of $P_\lambda$, the identity can be extended to $f \in B(\R^n)$.
\end{proof}

\begin{lemma}\label[lemma]{lem:multiplication_V}\sl If $V \in L^\infty (\R^n)$, then we have that
\begin{equation}\label{in:multiplication_V_bilinearform}
\Big| \int_{\R^n} V u v \, \Big| \leq \vvvert V \vvvert \| u \|_{B^\ast (\R^n)} \| v \|_{B^\ast (\R^n)}
\end{equation}
for all $u, v \in B^\ast (\R^n) $. Consequently, if $\vvvert V \vvvert < \infty$, then
\begin{equation}\label{in:boundedness_multiplicationV}
\| V u \|_{B (\R^n)} \leq \vvvert V \vvvert \| u \|_{B^\ast (\R^n)}
\end{equation}
for all $u \in B^\ast (\R^n) $.
\end{lemma}

\begin{proof}
Inequality \eqref{in:multiplication_V_bilinearform} is a trivial consequence of {the
Cauchy--Schwarz inequality}, while
the second inequality follows from the fact that $B^\ast(\R^n)$ is the dual space of $B(\R^n)$.
\end{proof}

The inequality \eqref{in:boundedness_multiplicationV} shows the boundedness from 
$B^\ast(\R^n)$ to $B(\R^n)$ of the operator
\[ u \in B^\ast(\R^n) \mapsto Vu \in B(\R^n). \]
Making an abuse of notation, $V$ denotes--- from now on---this operator as well as the 
potential. The context clarifies when we mean one or the other.

\begin{corollary}\label[corollary]{cor:I-Pinverse}\sl Let $V $ be in $L^\infty (\R^n)$ so that
$\vvvert V \vvvert < \infty$.
Using $C_n$ to denote the constant of \Cref{lem:p.v.},
we set $ \lambda_V = C_n \vvvert V \vvvert$.
Then, for every $\lambda > \lambda_V $,
the operator $\Id - P_\lambda \circ V$ has a bounded inverse in 
$B^\ast (\R^n)$.
Moreover, if this inverse is denoted by
$(\Id - P_\lambda \circ V)^{-1}$ we have that, for every $\delta > 0$, 
there exists an absolute constant $C > 0$ such that
\[ \| (\Id - P_\lambda \circ V)^{-1} \|_{\mathcal{L} (B^\ast(\R^n))} < C\]
for all $\lambda \geq \lambda_V + \delta$.
\end{corollary}

\begin{proof}
It is well known that if the Neumann series $\sum_{k=0}^\infty (P_\lambda \circ V)^k $ converges in 
$\mathcal{L} (B^\ast(\R^n))$, then $\Id - P_\lambda \circ V$ has a bounded inverse in 
$B^\ast (\R^n)$ and its inverse is given by
\begin{equation}
\label{id:Neumann_series}
(\Id - P_\lambda \circ V)^{-1} = \sum_{k=0}^\infty (P_\lambda \circ V)^k.
\end{equation}
In order to check that the Neumann series converges is enough to {verify}
\[\| P_\lambda \circ V \|_{\mathcal{L} (B^\ast(\R^n))} < 1. \]
Using \Cref{lem:p.v.,lem:multiplication_V}, we have that
\[ \| P_\lambda (Vu) \|_{B^\ast(\R^n)} \leq \frac{C_n}{\lambda} \| Vu \|_{B(\R^n)} \leq \frac{C_n}{\lambda} \vvvert V \vvvert \| u \|_{B^\ast (\R^n)}. \]
Then, for every $\lambda > \lambda_V $ we have that 
$ \| P_\lambda \circ V \|_{\mathcal{L} (B^\ast(\R^n))} < 1 $.
This implies that the Neumann series converges, which proves the first part of the statement.
To check the second part of the statement, one can use identity \eqref{id:Neumann_series}.
\end{proof}

\begin{proposition}\label[proposition]{prop:solutions}\sl Consider $V$  and $\lambda_V$ as in \Cref{cor:I-Pinverse}.
For $\lambda > \lambda_V$ and $f \in C^\infty (\Sph^{n - 1})$, set
\begin{align*}
u &= E_\lambda f,\\
v &= (\Id - P_\lambda \circ V)^{-1} [P_\lambda(Vu)].
\end{align*}
Then, $w = u + v \in B^\ast (\R^n)$ solves the equation
\[ (\Delta + \lambda^2 - V)w = 0 \enspace \textnormal{in} \enspace \R^n, \]
and, for every $\delta > 0$, there exists a constant $C> 0$, that {depends only on $n$}, such that
\[ \| v \|_{B^\ast (\R^n)} \leq \frac{C\lambda_V}{\lambda} \| u \|_{B^\ast (\R^n)} \]
for all $\lambda \geq \lambda_V + \delta$.
\end{proposition}

\begin{proof}
On the one hand, it follows from \Cref{lem:B*extension} that 
$E_\lambda f \in B^\ast(\R^n)$ and $(\Delta + \lambda^2) u = 0$ in $\R^n$. On the other hand,
as a consequence of \Cref{lem:p.v.,lem:multiplication_V} we know that
$P_\lambda (Vu) \in B^\ast(\R^n)$ and
\[ \| P_\lambda (Vu) \|_{B^\ast(\R^n)} \leq \frac{C_n}{\lambda} \vvvert V \vvvert \| u \|_{B^\ast (\R^n)}. \]
Then, by \Cref{cor:I-Pinverse}, the fact that $P_\lambda (Vu) \in B^\ast(\R^n)$ and the 
previous inequality we have that $v \in B^\ast(\R^n)$ for $\lambda > \lambda_V$, and
\[ \| v \|_{B^\ast (\R^n)} \leq \frac{C\lambda_V}{\lambda} \| u \|_{B^\ast (\R^n)} \]
for all $\lambda \geq \lambda_V + \delta$. Note that $v$ solves
\[(\Id - P_\lambda \circ V) v = P_\lambda(Vu),\]
and, by the last part of the statement of \Cref{lem:p.v.}, we have that
\[ (\Delta + \lambda^2 - V) v = Vu \enspace \textnormal{in} \enspace \R^n.  \]
Therefore, $w = u + v \in B^\ast (\R^n)$ and
\[ (\Delta + \lambda^2 - V)w = 0 \enspace \textnormal{in} \enspace \R^n. \]
This concludes the proof of this proposition.
\end{proof} 

\section{An orthogonality relation}\label{sec:orthogonality}
In this section we prove the orthogonality identity that yields 
\eqref{id:orthogonality_relation}. For this we will need a couple of lemmas{:
the first is an integration-by-parts formula and the second establishes an approximation of the identity in $B(\R^n)$.}

\begin{proposition}\label[proposition]{prop:orthogonality_time-independent} \sl 
Write $\Sigma = (0,T) \times \R^n$. Consider $V_1$ and $V_2$ in $L^\infty (\R^n)$ such that 
$\vvvert V_j \vvvert < \infty $ for $j \in \{1, 2 \}$ and let $\mathcal{U}_T^1 $ and $ \mathcal{U}_T^2$ denote their
corresponding initial-to-final-state maps. Then, if $\mathcal{U}_T^1 = \mathcal{U}_T^2$,
we have that
\[ \int_\Sigma (V_1 - V_2)u_1 \overline{v_2}\, = 0 \]
for all $u_1, v_2 \in C(\R; B^\ast (\R^n))$
solution{s} of the equations
\[ (i\partial_\tm + \Delta - V_1) u_1 =  (i\partial_\tm + \Delta - \overline{V_2}) v_2 = 0 \enspace \textnormal{in} \enspace \R \times \R^n. \]
\end{proposition}

In order to prove this orthogonality relation, we will need an integration-by-parts 
formula.

\begin{lemma}\label[lemma]{lem:integration-by-parts}\sl
For every $u\in C([0, T]; L^2(\R^n)) $ such that 
$(i\partial_\tm + \Delta) u $ belongs to $ L^1((0,T);B (\R^n)) $ and
$u(0,\centerdot) = u(T, \centerdot) = 0$,
and every $v \in C([0 , T]; B^\ast (\R^n))$ such that
$(i\partial_\tm + \Delta) v $ in $ L^1((0, T); B(\R^n))$, we have
\[\int_\Sigma (i\partial_\tm + \Delta) u \overline{v} = \int_\Sigma u \overline{(i\partial_\tm + \Delta)v} \,. \]
\end{lemma}

\begin{proof}[Proof of \Cref{prop:orthogonality_time-independent}]
We will prove it in two steps. The first one consists in
proving that the equality $\mathcal{U}_T^1 = \mathcal{U}_T^2$ implies that
\begin{equation}\label{id:orthogonality_physical-TH}
\int_\Sigma (V_1 - V_2)u_1 \overline{v_2}\, = 0 
\end{equation}
for every $u_1 \in C([0, T]; L^2(\R^n))$ solution of 
$(i\partial_\tm + \Delta - V_1) u_1 = 0$ in $\Sigma$, and every
$v_2 \in C(\R; B^\ast (\R^n))$ solution of 
$(i\partial_\tm + \Delta - \overline{V_2}) v_2 = 0 $ in $\R \times \R^n$.

For every solution $u_1 \in C([0, T]; L^2(\R^n))$,
let $w_2 \in C([0, T]; L^2(\R^n))$ be the solution
of the problem 
\[\left\{
		\begin{aligned}
		& (i\partial_\tm + \Delta -  V_2)w_2 = (V_1 - V_2)u_1 & & \textnormal{in} \, \Sigma, \\
		& w_2(0, \centerdot) = 0 &  & \textnormal{in} \, \R^n.
		\end{aligned}
\right.\]
Since $\mathcal{U}_T^1 = \mathcal{U}_T^2$ and $V_1, V_2 \in L^\infty (\R^n)$, we can
apply \cite[Lemma 4.4]{zbMATH07801151} with $F = (V_1 - V_2)u_1$ together with 
\cite[Proposition 4.1]{zbMATH07801151} to conclude that $w_2(T, \centerdot) = 0$.
By \Cref{lem:multiplication_V}, we know that 
$(V_1 - V_2) \overline{v_2} \in C(\R; B (\R^n))$. Since $B (\R^n) \subset L^2(\R^n)$, we 
know that $(V_1 - V_2)u_1 \overline{v_2}|_{\Sigma} \in L^1(\Sigma)$ and we can write
\[ \int_\Sigma (V_1 - V_2)u_1 \overline{v_2}\, = \int_\Sigma (i\partial_\tm + \Delta - V_2)w_2 \overline{v_2}\, . \]
Using the integration-by-parts formula of \Cref{lem:integration-by-parts} for $w_2$ and $v_2$ we have that
\[ \int_\Sigma (V_1 - V_2)u_1 \overline{v_2}\, = \int_\Sigma w_2 \overline{(i\partial_\tm + \Delta - \overline{V_2}) v_2}\, = 0. \]
In the last {equality} we have used the fact that $(i\partial_\tm + \Delta - \overline{V_2}) v_2=0$.
This {completes} the first
step of this proof.

In the second step, we show that the orthogonality relation 
\eqref{id:orthogonality_physical-TH} implies the one in the statement.
For $v_2 \in C(\R; B^\ast (\R^n))$ we know, by \Cref{lem:multiplication_V}, that 
$(V_1 - V_2) \overline{v_2} \in C(\R; B (\R^n))$. This means that
$(\overline{V_1} - \overline{V_2}) v_2|_\Sigma \in L^1((0, T); L^2(\R^n))$. Then,
let $w_1 \in C([0, T]; L^2(\R^n))$ be the solution of the problem
\[\left\{
		\begin{aligned}
		& (i\partial_\tm + \Delta - \overline{V_1})w_1 = (\overline{V_1} - \overline{V_2}) v_2|_\Sigma & & \textnormal{in} \, \Sigma, \\
		& w_1(T, \centerdot) = 0 &  & \textnormal{in} \, \R^n.
		\end{aligned}
\right.\]
We can apply \cite[Lemma 4.5]{zbMATH07801151} with 
$G = (\overline{V_1} - \overline{V_2}) v_2|_\Sigma $ together with 
the orthogonality relation \eqref{id:orthogonality_physical-TH} to conclude that
$w_1(0, \centerdot) = 0$.
Since $(V_1 - V_2)u_1|_{\Sigma} \overline{v_2}|_{\Sigma} \in L^1(\Sigma)$, we can write
\[ \int_\Sigma (V_1 - V_2)u_1 \overline{v_2}\, = \int_\Sigma u_1 \overline{(i\partial_\tm + \Delta - \overline{V_1}) w_1} \, = \int_\Sigma (i\partial_\tm + \Delta - V_1) u_1 \overline{w_1} \, = 0. \]
the integration-by-parts formula of \Cref{lem:integration-by-parts} for $w_1$ and 
$u_1$, and the fact that $u_1$ is solution. This {concludes} the second
step of the proof.
\end{proof}

\begin{proof}[Proof of \Cref{lem:integration-by-parts}]
Let us fix $u,v$ as in the statement of the lemma and $\chi$ a smooth bump function
in $\R^n$ such that $0\leq \chi \leq 1$ with $\chi\equiv 1$ on $\{x\in\R^n:\, |x|\leq 1\}$ and
$\supp \chi \subset \{x\in\R^n:\, |x|\leq 2\}$. Consider also a bump function
$\varphi\in \mathcal S(\R^n)$ with $\supp \varphi\subset \{x\in\R^n:\, |x|\leq 1\}$ and
$\int \varphi =1$. For $R\gg1$ and $0<\varepsilon\ll1$, and bump functions $\chi,\varphi$ as above, we define $\chi^R(x)\coloneqq \chi(x/R)$  and $\varphi_\varepsilon(x)\coloneqq \varepsilon^{-n}\varphi(x/\varepsilon)$, {and the approximation $v_{R,\ve}$ of $v$}
\[
v_{R,\varepsilon}(t,\centerdot)\coloneqq \chi^R(x)(v(t,\centerdot)*\varphi_\varepsilon)\eqqcolon \chi^R(x)v_\ve(t,x),\qquad (t,x)\in\Sigma;
\]
note that the convolution above is considered only in the spatial variables.  The Leibniz rule  yields
\[
\begin{split}
(i\partial_\tm +\Delta) v_{R,\ve} &= (i\partial_\tm +\Delta) ((v*\varphi_\ve) \chi_R) 
\\
&= \chi^R ((i\partial_\tm +\Delta)(v*\varphi_\ve)) + (v*\varphi_{\ve}) \Delta(\chi^R)+2\nabla (v*\varphi_\ve) \nabla\chi^R,
\end{split}
\]
the equality holding in $\mathcal D'(\Sigma)$. We record some easy estimates for the derivatives of the auxiliary functions $\chi^R,\varphi_\ve$: we have that
\[
\Delta(\chi^R)=R^{-2}(\Delta \chi)^R \ind_{\{R\leq |x|\leq 2R\}},\quad \nabla(\chi^R)=R^{-1}(\nabla \chi)^R \ind_{\{R\leq |x|\leq 2R\}} ,\quad \nabla(\varphi_\ve)=\ve^{-1}(\nabla\varphi)_\ve.
\]
Note that the functions $(\nabla\chi)^R,(\Delta \chi)^R$ are both smooth bump functions
which are supported in the annulus $\{x\in\R^n:\, R\leq|x|\leq 2R\}$ 
and have $L^\infty$-norms of the order $O_\chi(1)$, while $(\nabla\varphi)_\ve$ is a smooth bump function which is supported in $\{x\in\R^n:\, |x|\leq \ve\}$ and has $L^1(\R^n)$-norm of the order $O_\varphi(1)$.

We will need to estimate various norms of $v_{R,\ve}$ and $(i\partial_\tm+\Delta) v_{R,\ve}$ and the most flexible calculation will be to estimate the corresponding $L^2(D_j)$-norms for each $j\in \N_0$. To that end we fix some $t\in [0,T]$ and calculate
\[
\|v_{R,\ve}(t,\centerdot)\|_{L^2(D_j)}\leq \ind_{\{j:\,2^j\lesssim R\}}2^{j/2} \|v(t,\centerdot)\|_{B^*(\R^n)}.
\]
We continue with the estimates from the different summands produced by the Leibniz rule; we have
\[
\begin{split}
\|\chi^R (i\partial_\tm +\Delta) (v(t,\centerdot)*\varphi_\ve)\|_{L^2(D_j)}&\leq \| \varphi_\ve*(i\partial t+\Delta)v\|_{L^2(D_j)}\ind_{\{j:\,2^j\lesssim R\}}
\\
&\lesssim  \ind_{\{j:\,2^j\lesssim R\}} \int_{\R^n} |\varphi_\ve(y)| \left(\int_{D_j+y}|(i\partial_\tm+\Delta)v(t,\centerdot) |^2 \,\dd x \right)^{\frac12}  \dd y 
\\
&\lesssim \ind_{\{j:\,2^j\lesssim R\}} \|(i\partial_\tm+\Delta)v(t,\centerdot)\|_{L^2(D_{j-1}\cup D_j \cup D_{j+1})},
\end{split}
\]
where the last approximate inequality follows since $|y|\leq \ve$ for $y\in\supp \varphi_\ve$
and $0<\ve \ll 1\leq 2^j$ for $j\in\N_0$. For convenience, we consider here $D_{-1}\coloneqq \emptyset$. For the second summand produced by the Leibniz rule we can estimate
\[
\begin{split}
 \left\|(v(t,\centerdot)*\varphi_\ve) \Delta(\chi^R)\right\|_{L^2(D_j)}\lesssim \ind_{\{j:\,2^j\simeq  R\}}R^{-2}  \|v(t,\centerdot)\|_{L^2(\{R\leq|x|\leq 2R\})}\lesssim \ind_{\{j:\,2^j\simeq  R\}} R^{-\frac32}\|v(t,\centerdot)\|_{B^*(\R^n)}.
\end{split}
\]
Finally, for the third summand resulting from the Leibniz rule we write
\[
\begin{split}
\left\| \nabla (\varphi_\ve*v(t,\centerdot)) \nabla\chi^R\right\|_{L^2(D_j)} &\lesssim \ind_{\{j:\,2^j\simeq  R\}}R^{-1}\ve^{-1} \left\|(\nabla\varphi)_\ve*v(t,\centerdot)\right\|_{L^2(\{R\leq |x|\leq 2R\})}
\\
& \lesssim \ind_{\{j:\,2^j\simeq  R\}}R^{-\frac12}\ve^{-1}   \|v(t,\centerdot)\|_{B^*(\R^n)}.
\end{split}
\]
Using the estimates above we first estimate the $C([0,T];L^2(\R^n))$-norm of $v_{R,\ve}$ as follows
\[
\sup_{t\in[0,T]}\|v_{R,\ve}(t,\centerdot)\|_{L^2(\R^n)}\lesssim R^{\frac12} \|v(t,\centerdot)\|_{B^*(\R^n)  }.
\]
Next, we bound the $L^1((0,T);L^2(\R^n))$-norm of $(i\partial_\tm+\Delta)v_{R,\ve}$; combining the estimates above with the Leibniz rule yields
\[
\int_0 ^T \|(i\partial_\tm+\Delta)v_{R,\ve} \|_{L^2(\R^n)}\dd t \lesssim \int_0 ^T \|(i\partial_\tm+\Delta)v  \|_{B(\R^n)}\dd t + T(R^{-\frac32}+R^{-\frac12}\ve^{-1})\sup_{t\in[0,T]}\|v(t,\centerdot)\|_{B^*(\R^n)}.
\]
With these estimates in hand and using the assumptions of the lemma we have for each $R,\ve>0$ that $u,v_{R,\ve}\in C([0,T];L^2(\R^n))$ with $(i\partial_\tm+\Delta)u,(i\partial_\tm+\Delta){v_{R,\ve}}\in   L^1((0,T);L^2(\R^n))$. We can therefore appeal to \cite[Proposition 4.2]{zbMATH07801151} together with the assumption $u(0,\centerdot)=u(T,\centerdot)=0$ to get that
\[
\int_{\Sigma} (i\partial_\tm+\Delta)u \overline {v_{R,\ve}}=\int_{\Sigma}u \overline{(i\partial_\tm+\Delta)v_{R,\ve}}.
\]
Hereinafter we choose $R=R(\ve)\coloneqq \ve^{-3}$. In order to complete the proof of the lemma it will be enough to show that
\begin{align}
&\label{eq:rhs}\lim_{ \ve\to 0} \int_\Sigma \left[ u \overline{(i\partial_\tm+\Delta)(v_{R(\ve),\ve}-v)}\right]=0,
\\
&\label{eq:lhs}\lim_{ \ve\to 0} \int_\Sigma \left[ [(i\partial_\tm+\Delta)u] \overline {({v_{R(\ve),\ve}}-v)}\right]=0.
\end{align}
We begin with the proof of \eqref{eq:rhs}. Since $u\in C([0,T];L^2(\R^n))$ and $(i\partial_\tm+\Delta)v \in L^1((0,T);B(\R^n))\subset L^1((0,T);L^2(\R^n))$ and $\lim_{R\to\infty} (1-\chi^R)=0$, the dominated convergence theorem readily implies that
\[
\lim_{\ve\to 0}\left|\int_{\Sigma}u \overline{\left[(i\partial_\tm+\Delta)v -\chi^{R(\ve)}(i\partial_\tm+\Delta)v\right]}\right|=0.
\]
Next, we estimate the difference $\chi^{R(\ve)}\left[(i\partial_\tm+\Delta)v-(i\partial_\tm+\Delta)v_\ve\right]$; there holds
\[
\begin{split}
&\left|\int_{\Sigma}u \chi^{R(\ve)}\overline{\left[(i\partial_\tm+\Delta)v -(i\partial_\tm+\Delta)v_\ve\right]}\right|
\\
&\qquad\qquad \leq \sup_{t\in[0,T]}\|u(t,\centerdot)\|_{L^2(\R^n)}\int_0 ^T \left\| (i\partial_\tm+\Delta)v- [(i\partial_\tm+\Delta)v]*\varphi_\ve\right\|_{L^2(\R^n)}\dd t
\end{split}
\]
which converges to $0$ as $\ve\to 0$ since $(i\partial_\tm +\Delta)v\in L^1((0,T);B)\subset  L^1((0,T);L^2(\R^n))$ and $\varphi_\ve$ is an approximate identity. Thus, in order to prove \eqref{eq:rhs} it will suffice to control the term
\[
\chi^{R(\ve)}(i\partial_\tm+\Delta)v_\ve -v_{R(\ve),\ve}=v_\ve \Delta(\chi^{R(\ve)})+2\nabla(v_\ve)\nabla(\chi^{R(\ve)}),
\]
the equality in the display above following by the Leibniz rule. Taking into account the support of the bump function $\Delta(\chi^R)$, we estimate the first summand above as 
\[
\begin{split}
\left| \int_\Sigma u v_\ve \Delta(\chi^{R(\ve)})\right|&\lesssim T R(\ve)^{-2}\sup_{t\in[0,T]}\left(\|u(t,\centerdot)\|_{L^2(\R^n)} \left\| \varphi_\ve*v(t,\centerdot) \right\|_{L^2(R\leq |x|\leq 2R)}\right)
\\
&\lesssim TR(\ve)^{-3/2}\sup_{t\in[0,T]}\|u(t,\centerdot)\|_{L^2(\R^n)} \sup_{t\in[0,T]}\|v(t,\centerdot)\|_{B^*(\R^n)}
\end{split}
\]
which also tends to $0$ as $\ve\to 0$ since $u\in C([0,T];L^2(\R^n))$ and $v\in C([0,T];B^*(\R^n))$. Similarly, we have
\[
\begin{split}
& \left| \int_\Sigma u \nabla (v_\ve) \nabla (\chi^{R(\ve)})\right| \lesssim \frac{T}{\ve R(\ve)} \sup_{t\in[0,T]} \left(\|u(t,\centerdot)\|_{L^2(\R^n)} \left\| (\nabla \varphi)_\ve *v(t,\centerdot) \right\|_{L^2(\{R\leq |x|\leq 2R\})}\right)
\\
\quad&\lesssim\frac{T}{\ve R(\ve)}\sup_{t\in[0,T]} \|u(t,\centerdot)\|_{L^2(\R^n)}\|(\nabla\varphi)_{\ve}(y)*[v(t,\centerdot)\mathbf{1}_{\{|x|\simeq R\}}] \|_{L^2(\R^n)}
\\
\quad &\lesssim T R(\ve)^{-\frac12}\ve^{-1}\sup_{t\in[0,T]}\|u(t,\centerdot)\|_{L^2(\R^n)}\sup_{t\in[0,T]}\|v(t,\centerdot)\|_{B^*(\R^n)},
\end{split}
\]
where we used that $|y|\leq \ve\ll 1\ll R$ for $y\in\mathrm{supp}(\nabla\varphi)_\ve$ and thus, for $R\leq |x|\leq 2R$ we have $(\nabla \varphi)_\ve * v(t,\centerdot)\simeq (\nabla \varphi)_\ve * [v(t,\centerdot)\mathbf{1}_{\{|x|\simeq R\}}]$; the right hand side above tends to $0$ as $\ve\to 0$ since $R(\ve)^{-\frac12}\ve^{-1}=\ve^{\frac12}$ and this completes the proof of \eqref{eq:rhs}.

It remains to prove \eqref{eq:lhs}. Firstly, let us remark that for any $R,\ve>0$ we have
\[
\begin{split}
 &\int_0 ^T \int_{\R^n}\int_{\R^n} \left|[(i\partial_\tm +\Delta)u(t,x)] \chi^R(x) \varphi_\ve(x-y)v(t,y)\right|\,\dd y\,\dd x \,\dd t 
 \\
 &\qquad\qquad\qquad \lesssim \ve^{-1} \left(\int_{0} ^T \left\| (i\partial_\tm+\Delta)u\right\|_{B(\R^n)} \dd t \right) \sup_{t\in[0,T]}\|v(t,\centerdot)\|_{B^*(\R^n)}<\infty
\end{split}
\]
by the assumptions of the lemma. We can then use Fubini's theorem to justify the change of the order of integration below
\[
\int_\Sigma [(i\partial_\tm +\Delta) u] \overline{v_{R,\ve}}=\int_\Sigma \chi^R \left[\varphi_\ve*(i\partial_\tm+\Delta)u\right] \overline{v},
\]
where the convolution is again considered in the spatial variables only. Now we have 
\[
\begin{split}
&\left|\int_{\Sigma}\left[[(i\partial_\tm+\Delta)u] \overline{v_{R,\ve}}-[(i\partial_\tm+\Delta)u] \overline{v} \right]\right|=\left|\int_{\Sigma}\left[\chi^R\left[\varphi_\ve*\left((i\partial_\tm+\Delta)u\right)\right] -(i\partial_\tm+\Delta)u \right]\overline{v} \right|
\\
&\leq \left|\int_{\Sigma}(\chi^R-1) (i\partial_\tm+\Delta)u  \overline{v} \right|+\sup_{t\in[0,T]}\|v(t,\centerdot)\|_{B^*(\R^n)} \int_0 ^T \left\| (i\partial_\tm+\Delta)u -\varphi_\ve* \left((i\partial_\tm+\Delta)u\right)\right\|_{B(\R^n)}\dd t.
\end{split}
\]
The first summand in the display above tends to $0$ as $R\to \infty$ by dominated convergence 
since our assumptions imply that $[(i\partial_\tm+\Delta)u] \overline{v}\in L^1(\Sigma)$. The 
second summand in the right hand side of the display above tends to $0$ as $\ve\to 0$ as a 
consequence of \Cref{lem:appidB}---proven below---since 
$(i\partial_\tm+\Delta)u\in L^1((0,T);B(\R^n))$ and $\varphi$ is an $L^1$-rescaled smooth bump 
function. This completes the proof of \eqref{eq:lhs} and with that the proof of the lemma.
\end{proof}

The following lemma shows that approximate identities converge in $L^1((0,T);B(\R^n))$. Its proof is rather elementary but we include it here for the convenience of the reader.

\begin{lemma}\label[lemma]{lem:appidB}\sl For $\varphi\in L^1(\R^n)$ such that $\int_{\R^n}\varphi =1$ and $\int_{\R^n}|y|^{1/2}|\varphi(y)|\,\dd y<+\infty$. For $\ve>0$ define the $L^1$-rescaled bump function $\varphi_\ve(x)\coloneqq \ve^{-n}\varphi(x/\ve)$. For every $U\in L^1((0,T);B(\R^n))$ there holds
\[
\lim_{\ve\to 0} \int_0 ^T \left\| U(t,\centerdot)-\varphi_\ve * U(t,\centerdot) \right\|_{B(\R^n)}\, \dd t=0,
\]
with the convolution above being considered in the spatial variables only.
\end{lemma}

\begin{proof} Let us fix $\delta>0$ and $U$ as above.  First of all notice that by an easy calculation we have for each $j\in\mathbb N_0$ that for each $t\in[0,T]$, there holds
\[
\left\| U(t,\centerdot)*\varphi_\ve\right\|_{L^2(D_j)}\leq \int_{\R^n}|\varphi_\ve (y)| \left(\int_{D_j}\left|U(t,x-y)\right|^2\,\dd x\right)^{\frac12} \, \dd y.
\]
Note that for $x\in D_j$ and $|y|<100^{-1}2^j$ we have that $|x-y|\simeq |x|$, so that $x-y\in D_{j-1}\cup D_j \cup D_{j+1}$. We then get
\[
\begin{split}
\left\| U(t,\centerdot)*\varphi_\ve\right\|_{B(\R^n)}&\lesssim \int_{\R^n}|\varphi_\ve(y)| \sum_{2^j>100 |y|} 2^{j/2} \left(\int_{D_{j-1}\cup D_j\cup D_{j+1}}\left|U(t,x)\right|^2\,\dd x\right)^{\frac12} \, \dd y
\\
&\qquad +\left\| U(t,\centerdot)\right\|_{L^2(\R^n)} \int_{\R^n}|y|^{\frac12}|\varphi_\ve(y)|\dd y  
\\
& \lesssim_\varphi \left\|U(t,\centerdot)\right\|_{B(\R^n)}+\ve^{\frac12}\left\|U(t,\centerdot)\right\|_{L^2(\R^n)}\lesssim \left\|U(t,\centerdot)\right\|_{B(\R^n)}
\end{split}
\]
uniformly in $\ve\leq 1$. Integrating for $t\in(0,T)$ gives the estimate
\[
\int_0 ^T \left\| U(t,\centerdot)*\varphi_\ve\right\|_{B(\R^n)}\dd t\lesssim_\varphi\int_0 ^T \left\| U(t,\centerdot)\right\|_{B(\R^n)}\dd t,\qquad 0<\ve <1
\]

Now a similar calculation shows that
\[
\begin{split}
&\int_0 ^T\left\| U(t,\centerdot)-\varphi_\ve * U(t,\centerdot) \right\|_{B(\R^n)}\dd t=\int_0 ^T\sum_{j\in\N_0}2^{j/2}\left\|U(t,\centerdot)-\varphi_\ve*U(t,\centerdot)\right\|_{L^2(D_j)} \dd t
\\
& \qquad \lesssim \int_{\R^n}|\varphi_\ve(y)| \int_0 ^T \sum_{2^j>100|y|}2^{j/2} \left\| U(t,\centerdot-y)-U(t,\centerdot)\right\|_{L^2(D_j)}\dd t
\\
&\qquad\qquad\qquad+\ve^{\frac12}\left(\int_0 ^T \left\|U(t,\centerdot)\right\|_{B(\R^n)}\dd t\right)\int_{\R^n}|y|^{\frac12}|\varphi(y)|\,\dd y
\\
&\qquad \eqqcolon \mathrm{I}(\ve)+\mathrm{II}(\ve).
\end{split}
\]
Since $\lim_{\ve\to 0}\mathrm{II}(\ve)=0$ it suffices to deal with the term $\mathrm{I}(\ve)$ which we split further as follows: for some $\zeta >0$ to be determined momentarily, we write
\[
\begin{split}
\mathrm{I}(\ve)&\leq  \int_{|y|<\zeta}|\varphi_\ve(y)| \left(\int_0 ^T \left\|U(t,\centerdot-y)-U(t,\centerdot)\right\|_{B(\R^n)} \dd t \right)\dd y 
\\
&\qquad\qquad+\left(\int_{|y|\geq \zeta}|\varphi_\ve(y)|\,\dd y\right) \int_0 ^T \left\| U(t,\centerdot)\right\|_{B(\R^n)} \dd t
\\
&\lesssim_\varphi \sup_{|y|<\zeta}\left(\int_0 ^T \left\|U(t,\centerdot-y)-U(t,\centerdot)\right\|_{B(\R^n)}\dd t \right)+\left(\int_{|y|\geq \zeta}|\varphi_\ve(y)|\,\dd y\right)\int_0 ^T \left\| U(t,\centerdot)\right\|_{B(\R^n)} \dd t.
\end{split}
\]
For the first summand above we use that translations are continuous in $L^1((0,T);B(\R^n))$. Thus there exists $\zeta_0>0$ such that \[
\sup_{|y|<\zeta_0}\left(\int_0 ^T \left\|U(t,\centerdot-y)-U(t,\centerdot)\right\|_{B(\R^n)} \dd t\right)<\delta.
\]
Fixing this $\zeta=\zeta_0$, the second summand above tends to $0$ as $\ve\to 0$ by the assumptions on $\varphi$. 

Finally, let us show that translations are continuous in $L^1((0,T);B(\R^n))$.  Given $\eta>0$ and a translation vector $y\in\R^n$ with $|y|<1/2$ we have 
\[
\int_0 ^T \sum_{2^j>R} 2^{j/2}\left\|U(t,\centerdot-y)-U(t,\centerdot)\right\|_{L^2(D_j)}\dd t\lesssim \int_0 ^T \sum_{2^j>R} 2^{j/2}\left\|U(t,\centerdot)\right\|_{L^2(D_{j-1}\cup D_j \cup D_{j+1})} \dd t \lesssim \eta
\]
if $R$ is sufficiently large depending on $\eta$. On the other hand
\[
\int_0 ^T \sum_{2^j\leq R} 2^{j/2}\left\|U(t,\centerdot-y)-U(t,\centerdot)\right\|_{L^2(D_j)}\dd t\lesssim R^{1/2}\int_0 ^T\left\|U(t,\centerdot-y)-U(t,\centerdot)\right\|_{L^2(\R^n)}\dd t\lesssim \eta
\]
whenever $|y|$ is sufficiently small depending on {$R$ and} $\eta$, using that translations are continuous in the space $L^1((0,T);L^2(\R^n))$.
\end{proof}

\section{Proof of \texorpdfstring{\Cref{th:bounded}}{th:bounded}}\label{sec:uniqueness}
In this section we prove \Cref{th:bounded}. To do so, we use the orthogonality relation of \Cref{prop:orthogonality_time-independent}, proved in
 \Cref{sec:orthogonality}. We will plug into it solutions to the Schr\"odinger equation of the form $u(t,x)=e^{-i\lambda^2 t} w(x)$, with the stationary states $w$ being as those constructed in \Cref{sec:stationary_states}. The goal is, on the one hand, to check that the most useful 
information of the stationary states are contained in the Herglotz waves, and on the other 
hand, to see that from the Herglotz waves we can compute the Fourier transform of $V_1 - V_2$.

We start by collecting some properties about the densities we use to define the Herglotz waves
and then we curry the analysis to check the Fourier transform of $V_1 - V_2$ vanishes.

\subsection{Appropriate densities on \texorpdfstring{$\Sph^{n-1}$}{S}}
Consider $\chi \in \mathcal{D} (\R^n)$ such that $0 \leq \chi (\xi) \leq 1$ for 
all $\xi \in \R^n$ and $\supp \chi \subset \{ \xi \in \R^n : |\xi| < 1/2 \}$.
For $\varepsilon > 0$, define
\[ \chi_\varepsilon (\xi) = \frac{1}{\varepsilon^{n - 1}} \chi (\xi^\prime/\varepsilon, (\xi_n - 1)/\varepsilon^2) \quad \forall \xi \in \R^n. \]
Here we use the notation $\xi^\prime = (\xi \cdot e_1, \dots, \xi \cdot e_{n-1})$ and 
$\xi_n = \xi  \cdot e_n$ with $\{ e_1, \dots, e_n \}$ denoting the standard basis of
$\R^n$. We are interested in the restriction of this function to $\Sph^{n - 1}$:
\[ f^{e_n}_\varepsilon = \chi_\varepsilon|_{\Sph^{n - 1}}.\]
We have taken a parabolic scaling in the definition of $\chi_\varepsilon$ to ensure
that the support $f^{e_n}_\varepsilon$ is \textit{like} a geodesic ball in 
$\Sph^{n - 1}$ of radius $\varepsilon$ and centred at $e_n$. Finally, if 
$Q \in \Orth (n)$, we define
\[ f^{Qe_n}_\varepsilon (\theta) = f^{e_n}_\varepsilon (Q^\T \theta) \quad \forall \theta \in \Sph^{n-1}. \]

\begin{lemma}\label[lemma]{lem:f_epsilon} \sl For every $Q \in \Orth (n)$, we have that
\[ \lim_{\varepsilon \to 0} \| f^{Qe_n}_\varepsilon \|_{L^1(\Sph^{n - 1})} =  \int_{\R^{n-1}} \chi(\eta, - |\eta|^2/2) \dd \eta. \]
Furthermore, there exists a constant $C > 0$ that only depends on $n$ such that
\[\| f^{Qe_n}_\varepsilon \|_{L^2(\Sph^{n-1})} \leq \frac{C}{\varepsilon^{(n-1)/2}}\]
for all $\varepsilon \in (0, 1]$ and $Q \in \Orth(n)$.
\end{lemma}

\begin{proof}
Since the surface measure $\sigma$ in $\Sph^{n-1}$ is invariant under transformation{s} in $\Orth(n)$
we have that
\[ \| f^{Qe_n}_\varepsilon \|_{L^p(\Sph^{n-1})}^p = \| f^{e_n}_\varepsilon \|_{L^p(\Sph^{n-1})}^p = \frac{1}{\varepsilon^{p(n - 1)}} \int_{\R^{n-1}} \chi \Big(\frac{\eta}{\varepsilon}, \frac{(1 - |\eta|^2)^{1/2} - 1}{\varepsilon^2} \Big)^p \frac{1}{(1 - |\eta|^2)^{1/2}} \dd \eta, \]
for $p \in \{ 1, 2 \}$. Multiplying and dividing by the quantity $1 + (1 - |\eta|^2)^{1/2}$
we have that
\[\frac{(1 - |\eta|^2)^{1/2} - 1}{\varepsilon^2} = - \frac{|\eta/\varepsilon|^2}{1 + (1 - |\eta|^2)^{1/2}}\]
Then, after the change of variable $\eta = \varepsilon \kappa$ we obtain
\[\| f^{Qe_n}_\varepsilon \|_{L^p(\Sph^{n-1})}^p = \varepsilon^{(1-p)(n - 1)} \int_{\R^{n-1}} \chi \Big(\kappa, - \frac{|\kappa|^2}{1 + (1 - |\varepsilon\kappa|^2)^{1/2}} \Big)^p \frac{1}{(1 - |\varepsilon\kappa|^2)^{1/2}} \dd \kappa. \]
Since
\[\lim_{\varepsilon \to 0} \int_{\R^{n-1}} \chi \Big(\kappa, - \frac{|\kappa|^2}{1 + (1 - |\varepsilon\kappa|^2)^{1/2}} \Big)^p \frac{1}{(1 - |\varepsilon\kappa|^2)^{1/2}} \dd \kappa = \int_{\R^{n-1}} \chi (\kappa, - |\kappa|^2/2)^p \dd \kappa, \]
we can derive the conclusions announced in the statement for $p \in \{ 1, 2 \}$.
\end{proof}

\subsection{The Fourier transform of \texorpdfstring{$V_1 - V_2$}{V} vanishes} Given $\kappa \in \R^n$ consider $\nu \in \Sph^{n-1}$ such that $\kappa \cdot \nu = 0$.
For $\lambda \geq |\kappa|/2$ we define
\begin{align*}
\omega_1 = \frac{1}{\lambda} \frac{\kappa}{2} + \Big( 1 - \frac{|\kappa|^2}{4 \lambda^2} \Big)^{1/2} \nu,\\
\omega_2 = \frac{1}{\lambda} \frac{\kappa}{2} - \Big( 1 - \frac{|\kappa|^2}{4 \lambda^2} \Big)^{1/2} \nu.
\end{align*}
Note that $\omega_1$ and $\omega_2$ belong to $\Sph^{n-1}$. Let $Q_1$ and $Q_2$ denote 
two matrices in $\Orth(n)$ so that $\omega_j = Q_j e_n$ for $j \in \{ 1, 2 \}$. We 
construct the stationary states $w_j = u_j + v_j$ of \Cref{prop:solutions} with
\begin{align*}
u_j &= E_\lambda f^{Q_j e_n}_\varepsilon,\\
v_j &= (\Id - P_\lambda \circ V_j)^{-1} [P_\lambda(V_j u_j)];
\end{align*}
{here} $\lambda > \max(|\kappa|/2, \lambda_{V_j})$ and {$0<\varepsilon \leq 1$}.
Recall that, for every $\delta > 0$ there is a constant $C>0$, that {depends only on} $n$,
such that
\begin{equation}\label{in:decay_remainder_B*}
\| v_j \|_{B^\ast (\R^n)} \leq \frac{C\lambda_{V_j}}{\lambda} \| u_j \|_{B^\ast (\R^n)}
\end{equation}
for all $\lambda \geq \max(|\kappa|/2, \lambda_{V_j} + \delta)$ and $\varepsilon > 0$.
Note that the functions
\begin{align*}
& (t, x) \in \R \times \R^n \longmapsto e^{-i\lambda^2 t} w_1(x) \in \C,\\
& (t, x) \in \R \times \R^n \longmapsto e^{-i\lambda^2 t} \overline{w_2(x)} \in \C
\end{align*}
are solutions to the Schr\"odinger equation with potentials $V_1$ and $\overline{V_2}$, respectively.
Plugging them into the orthogonality 
relation stated in \Cref{prop:orthogonality_time-independent},  {and integrating in time}, we obtain that
\begin{equation}\label{id:leading=reminder}
\int_{\R^n} (V_1 - V_2) u_1 u_2\, = - \int_{\R^n} (V_1 - V_2) [u_1 v_2 + v_1 u_2 + v_1 v_2]\,
\end{equation}
for all {$\lambda > \max(|\kappa|/2, \lambda_{V_1}+\delta, \lambda_{V_2}+\delta)$}
and {$0<\varepsilon \leq 1$.}

We start estimating the right-hand side of the identity \eqref{id:leading=reminder}.
By the inequalities \eqref{in:multiplication_V_bilinearform} and 
\eqref{in:decay_remainder_B*}, and \Cref{lem:B*extension},
we can conclude that, for every $\delta > 0$
there is a constant $C>0$, that only depends on $n$ and the quantity
$ \vvvert V_1 \vvvert + \vvvert V_2 \vvvert $, such that
\[\Big| \int_{\R^n} (V_1 - V_2) [u_1 v_2 + v_1 u_2 + v_1 v_2]\, \Big| \leq C \Big(  \frac{1}{\lambda^n} + \frac{1}{\lambda^{n+1}} \Big) \| f^{Q_1 e_n}_\varepsilon \|_{L^2(\Sph^{n-1})} \| f^{Q_2 e_n}_\varepsilon \|_{L^2(\Sph^{n-1})} \]
for all $\lambda \geq \max(|\kappa|/2, \lambda_{V_1}+\delta, \lambda_{V_2}+\delta)$
and $\varepsilon > 0$.

Next, we turn our attention to the left-hand side of the identity 
\eqref{id:leading=reminder}. For convenience, we call $ F = V_1 - V_2 $.
Since $V_1$ and $V_2$ belong to $L^1 (\R^n)$, we can apply 
Fubini's theorem to write
\[ \int_{\R^n} F u_1 u_2\, = (2\pi)^{n/2} \int_{\Sph^{n-1} \times \Sph^{n-1}} \widehat{F} \big( \lambda (\theta + \omega) \big) f^{Q_1 e_n}_\varepsilon (\theta) f^{Q_2 e_n}_\varepsilon(\omega) \, \dd \mu (\theta, \omega), \]
where $\mu = \sigma \otimes \sigma$. With the choices of $\omega_j$ and $Q_j$ with
$j \in \{ 1, 2 \}$ we can write
\begin{align*}
&\frac{1}{(2\pi)^{n/2}} \int_{\R^n} F u_1 u_2\, = \widehat{F}(\kappa) \| f^{Q_1 e_n}_\varepsilon \|_{L^1(\Sph^{n-1})} \| f^{Q_2 e_n}_\varepsilon \|_{L^1(\Sph^{n-1})} \\
& \qquad + \int_{\Sph^{n-1} \times \Sph^{n-1}} \big[ \widehat{F} \big( \lambda (\theta + \omega) \big) - \widehat{F} \big( \kappa \big) \big] f^{Q_1 e_n}_\varepsilon (\theta) f^{Q_2 e_n}_\varepsilon(\omega) \, \dd \mu (\theta, \omega).
\end{align*}
The second summand of the right-hand side of the previous identity
can be bounded by
\[ \sup_{(\theta, \omega) \in K^1_\varepsilon \times K^2_\varepsilon} \big| \widehat{F} \big( \lambda (\theta + \omega) \big) - \widehat{F} \big( \kappa \big) \big| \| f^{Q_1 e_n}_\varepsilon \|_{L^1(\Sph^{n-1})} \| f^{Q_2 e_n}_\varepsilon \|_{L^1(\Sph^{n-1})} \]
where $K^j_\varepsilon$ denotes the support of $f^{Q_j e_n}_\varepsilon$.
Note that
\[ \widehat{F} \big( \lambda (\theta + \omega) \big) - \widehat{F} \big( \kappa \big) =
\frac{1}{(2\pi)^{n/2}} \int_{\R^n} [ e^{i [\kappa - \lambda(\theta + \omega)] \cdot x} - 1] e^{-i \kappa \cdot x} F(x) \, \dd x, \]
and
\[ \kappa - \lambda(\theta + \omega) = \lambda [(Q_1 e_n - \theta) + (Q_2 e_n - \omega)]. \]
Additionally,
\[ (\theta, \omega) \in K^1_\varepsilon \times K^2_\varepsilon \Rightarrow \big| \lambda [(Q_1 e_n - \theta) + (Q_2 e_n - \omega)] \big| < \lambda \varepsilon. \]
Hence,
\[\sup_{(\theta, \omega) \in K^1_\varepsilon \times K^2_\varepsilon} \big| \widehat{F} \big( \lambda (\theta + \omega) \big) - \widehat{F} \big( \kappa \big) \big| \leq \sup_{|\xi| < \lambda\varepsilon} \Big| \frac{1}{(2\pi)^{n/2}} \int_{\R^n} [ e^{i \xi \cdot x} - 1] e^{-i \kappa \cdot x} F(x) \, \dd x \Big|. \]
Introducing the function
\[\gamma (\rho) = \frac{1}{(2\pi)^{n/2}} \int_{\R^n} \sup_{|\xi| < \rho} | e^{i \xi \cdot x} - 1| |F(x)| \, \dd x \quad \forall \rho \in (0, \infty), \]
we can write 
\begin{align*}
\Big| \int_{\Sph^{n-1} \times \Sph^{n-1}} \big[ \widehat{F} \big( \lambda (\theta + \omega) \big) - \widehat{F} \big( \kappa \big) \big] \, & f^{Q_1 e_n}_\varepsilon (\theta) f^{Q_2 e_n}_\varepsilon(\omega) \, \dd \mu (\theta, \omega) \Big|  \\
& \leq \gamma ( \lambda \varepsilon ) \| f^{Q_1 e_n}_\varepsilon \|_{L^1(\Sph^{n-1})} \| f^{Q_2 e_n}_\varepsilon \|_{L^1(\Sph^{n-1})}
\end{align*}
for all $\lambda \geq |\kappa|/2$ and $\varepsilon > 0$.
It is convenient to note at this point that, since $F \in L^1(\R^n)$, we can conclude by 
the {dominated} convergence theorem that
\[ \lim_{\rho \to 0} \gamma (\rho) = 0. \]

Gathering all this information and using \Cref{lem:f_epsilon}, we can ensure that,
for every $\delta > 0$, there is a constant $C>0$, that only depends on 
$n$, $|\kappa|$, and the quantity $ \vvvert V_1 \vvvert + \vvvert V_2 \vvvert $, such that
\[ |\widehat{F}(\kappa)| \leq \gamma ( \lambda \varepsilon ) + \frac{C}{\lambda^n \varepsilon^{n-1}} \]
for all $\lambda \geq \max(|\kappa|/2, \lambda_{V_1}+\delta, \lambda_{V_2}+\delta)$
and $\varepsilon > 0$. Choosing 
$\varepsilon = \lambda^{-(1+1/n)}$ the previous inequality takes the form
\[ |\widehat{F}(\kappa)| \leq \gamma ( \lambda^{-1/n} ) + \frac{C}{\lambda^{1/n}}. \]
{Letting} $\lambda$ {go} to infinity, we obtain that $\widehat{F}(\kappa) = 0$. Since
$\kappa$ was an arbitrary vector of $\R^n$ we can conclude that the Fourier transform
of $F$ is identically zero. This means that $V_1 = V_2$ {$\mathrm{a.e.}$}, which {completes the proof of} 
\Cref{th:bounded}.

\sloppy
\begin{acknowledgements}
All the authors were supported by the grant PID2021-122156NB-I00 funded by 
MICIU/AEI/10.13039/501100011033 and FEDER, UE.
Additionally, P.C. and T.Z. were also funded by BCAM-BERC 2022-2025
and Severo Ochoa CEX2021-001142-S. I.P. was also supported by grant IT1615-22 of the Basque 
Government. Finally P.C. and I.P. are also funded by Ikerbasque, the Basque Foundation for 
Science.
\end{acknowledgements}

\bibliography{references}{}
\bibliographystyle{plain}

\end{document}